%% file: arxiv.tex
\documentclass[11pt]{amsart}
\usepackage[english]{babel}
\usepackage[toc,page]{appendix}
\usepackage{bbm}
\usepackage{fullpage}
\usepackage[top=1in, bottom=1in, left=1in, right=1in]{geometry}
\usepackage{etaremune}
\usepackage{enumitem}

\usepackage{witharrows}

\usepackage{amssymb,amsmath,amsthm}

\usepackage{hyperref}

\usepackage{mathtools}
\mathtoolsset{showonlyrefs}

\DeclareSymbolFont{rsfs}{U}{rsfs}{m}{n}
\DeclareSymbolFontAlphabet{\mathscrsfs}{rsfs}
\usepackage{mathrsfs}

\usepackage{url}

\hypersetup{
  colorlinks   = true, 
  urlcolor     = blue, 
  linkcolor    = blue, 
  citecolor   = red 
}

\usepackage[labelformat=empty]{caption}

\newtheorem{theorem}{Theorem}[section]
\newtheorem{lemma}[theorem]{Lemma}

\newtheorem{proposition}[theorem]{Proposition}
\newtheorem{corollary}[theorem]{Corollary}

\theoremstyle{definition}
\newtheorem{definition}{Definition}

\numberwithin{equation}{section}

\usepackage{times}

\input{commands}

\date{}

\begin{document}

\title{Tight Low Degree Hardness for Optimizing Pure Spherical Spin Glasses}

\author{Mark Sellke}
\address{Department of Statistics, Harvard University} 
\email{msellke@fas.harvard.edu}

\maketitle

\begin{abstract}
    \noindent 
    We prove constant degree polynomial algorithms cannot optimize pure spherical $p$-spin Hamiltonians beyond the algorithmic threshold $\ALG(p)=2\sqrt{\frac{p-1}{p}}$.
    The proof goes by transforming any hypothetical such algorithm into a Lipschitz one, for which hardness was shown previously by the author and B.\ Huang.
\end{abstract}

\section{Introduction}

Let $\bG_N \in \lt(\bbR^N\rt)^{\otimes p}$ be an order $p$ tensor with IID standard Gaussian entries $g_{i_1,\dots,i_p}\sim \cN(0,1)$, where $p\geq 3$ is fixed and $N$ is large.
The pure $p$-spin Hamiltonian $H_N$ is the random polynomial
\begin{equation}
    \label{eq:def-hamiltonian}
    H_N(\bsig)
    =
    N^{-\frac{p-1}{2}}
    \la \bG_N, \bsig^{\otimes p} \ra
    =
    N^{-\frac{p-1}{2}}
    \sum_{i_1,\dots,i_p=1}^N g_{i_1,\dots,i_p}\sigma_{i_1}\dots\sigma_{i_p}
    .
\end{equation}
An equivalent definition is that $H_N$ is a centered Gaussian process on $\bbR^N$ with covariance
\begin{equation}
    \label{eq:def-xi}
\bbE[H_N(\bsig)H_N(\brho)]
=
N
\Big(\frac{\la \bsig,\brho\ra}{N}\Big)^p.
\end{equation}
We study the spherical $p$-spin model, taking $\cS_N=\{\bsig\in\bbR^N~:~\|\bsig\|=\sqrt{N}\}$ to be the domain of $H_N$.

We will be interested in the efficient optimization of $H_N$.
The in-probability limit of the global maximum value as $N\to\infty$ 
\begin{equation}
\label{eq:OPT-def}
\OPT(p)
=
\plim_{N\to\infty}
\max_{\bsig\in\cS_N}H_N(\bsig)/N
\end{equation}
is given by the celebrated Parisi formula \cite{parisi1979infinite,talagrand2006parisi,talagrand2006free,panchenko2013parisi,chen2013aizenman}.
However, it has long been predicted by physicists that Langevin dynamics and other algorithms fail to come anywhere close to reaching this value, due to the extreme non-convexity of the landscape (see e.g.\ \cite{auffinger2013complexity}).

Many computational hardness results for this and closely related random optimization problems have emerged in the past decade, based on a geometric framework known as the overlap gap property \cite{gamarnik2021survey}. 
This method was introduced in \cite{gamarnik2017limits} to show that local algorithms cannot find near-maximal independent sets in sparse random graphs.
It has since been refined significantly \cite{rahman2017independent,wein2020independent} and been applied to show hardness for Max-$k$-SAT \cite{gamarnik2017performance,bresler2021ksat}, densest submatrix \cite{gamarnik2018finding}, the number partitioning problem \cite{gamarnik2021partitioning}, the random perceptron \cite{gamarnik2022algorithms,gamarnik2023geometric,li2024discrepancy}.
In all settings, one considers classes of algorithms which are \emph{stable} to small perturbations in the input (in this case $\bG_N$) and argues that such algorithms rarely succeed simultaneously on a gradually changing sequence of inputs.
For pure mean-field spin glasses, \cite{gamarnik2019overlap,gamarnik2020optimization,huang2025strong} have shown \emph{low-degree hardness} for achieving energy $\OPT$.
In particular the most recent of these showed that for $\eps=\eps(p)>0$ depending only on $p$, degree $o(N)$ polynomial algorithms have probability $o(1)$ to reach objective value $\OPT(p)-\eps$ when $p\geq 4$ is even (stated for the Ising $p$-spin model where $\bsig\in\{\pm 1\}^N$).

The results mentioned above show computational hardness of reaching asymptotic optimality, often with relatively sharp quantitative bounds, but are generally unable to characterize the \emph{exact} computational limit of efficient algorithms.
For the spherical $p$-spin model, physicists have in fact predicted (see e.g.\ \cite{cugliandolo1994out,biroli1999dynamical}) an exact threshold energy
\[
\ALG(p)=2\sqrt{\frac{p-1}{p}}<\OPT(p),
\]
for optimization.
This belief has also seen a flurry of recent progress.
It is now known that $\ALG(p)$ is indeed achievable using Hessian ascent \cite{subag2018following}, approximate message passing \cite{montanari2021optimization,ams20}, and low-temperature Langevin dynamics \cite{sellke2023threshold}.
In the opposite direction, the author and B.\ Huang \cite{huang2021tight,huang2023algorithmic} introduced the \emph{branching overlap gap property} to show that algorithms with dimension-free Lipschitz constant (including gradient descent and Langevin dynamics for short time) cannot exceed $\ALG(p)$.
Similar results hold for the more general mixed $p$-spin models, which in the Ising case also yield exact thresholds in certain Max-CSPs \cite{jones2023random}.

The hardness results of \cite{huang2021tight,huang2023algorithmic} are exact, but apply only to Lipschitz algorithms, which are more restricted than general stable algorithms (as discussed further in Subsection~\ref{subsec:discussion}).
In particular, stable algorithms include low-degree polynomials, which are often considered a good proxy for all polynomial time algorithms.
Nonetheless, in random optimization problems with no planted signal, it appears likely that Lipschitz or local algorithms suffice to achieve the best performance attainable by any polynomial time algorithm.
A natural route to provide evidence for this belief is to show that intermediate algorithm classes cannot outperform the exact thresholds for Lipschitz algorithms established by \cite{huang2021tight,huang2023algorithmic}.

Our main result shows that for optimizing spherical $p$-spin glasses, low-degree polynomials indeed cannot surpass the threshold $\ALG(p)$.
Namely for any $E>\ALG(p)$, degree $O(1)$ polynomials have probability $o_{N\to\infty}(1)$ to achieve energy at least $E$.

\subsection{Statement of Results}

For our purposes, an \emph{algorithm} is a measurable function $\cA_N:\sH_N\times\Omega_N\to\cB_N$, where $\cB_N=\{\bsig\in\bbR^N~:~\|\bsig\|\leq\sqrt{N}\}$ is the convex hull of $\cS_N$, and $\sH_N=\bbR^{N^p}$ is the state space for $H_N$ (identified with its coefficient tensor $\bG_N$), and $\Omega_N$ is an arbitrary Polish space.
(Note that since $H_N$ is homogeneous, extending the codomain from $\cS_N$ to $\cB_N$ does not impact e.g.\ the value in \eqref{eq:OPT-def}.)
The output of the algorithm is $\cA_N(H_N,\omega)$ where $\omega\in\Omega_n$ is drawn independently of $H_N$ (thus allowing $\cA_N$ to be randomized).
We first recall existing results on Lipschitz hardness, for both comparison and later use.
$\cA_N$ is $L$-Lipschitz if for each fixed $\omega$, the map $H_N\mapsto \cA_N:(H_N,\omega)$ is $L$-Lipschitz; here $H_N$ is identified with its $p$-tensor of coefficients and metrized using the un-normalized Euclidean distance (see \eqref{eq:hamiltonian-norm}).

\begin{proposition}[{\cite{huang2021tight,huang2023algorithmic}}]
\label{prop:BOGP-input}
Fix an integer $p\geq 3$ and let $H_N$ be a pure spherical $p$-spin Hamiltonian.
For any $L>0$ and $E>\ALG(p)$ there exists $c>0$ such that if $\cA_N$ is an $L$-Lipschitz algorithm, then
\[
\bbP[
H_N(\cA_N(H_N))/N \geq E]
\leq 
e^{-cN}.
\]
\end{proposition}

The main new result of this paper establishes hardness at the same threshold $\ALG(p)$ for a larger class of \emph{stable} algorithms, which includes both Lipschitz algorithms and low-degree polynomials.
To define this class, we say a pair $(H_N,\wt H_N)$ of $p$-spin Hamiltonians (each with marginal law as in \eqref{eq:def-hamiltonian}) to be $q$-correlated for $q\in [0,1]$ if their joint law is given by $\wt H_N=qH_N+\sqrt{1-q^2}H_N'$ for $H_N'$ an IID copy of $H_N$.

\begin{definition}
\label{def:stability}
Given $S>0$ and with $H_N,H_{N,1-\eps}$ denoting $(1-\eps)$-correlated $p$-spin Hamiltonians, we say $\cA_N$ is $(S,\eps)$-stable if
\[
\bbE[\|\cA_N(H_N)-\cA(H_{N,1-\eps})\|^2/N]\leq S\eps.
\]
Further, $\cA_N$ is $S$-stable if it is $(S,\eps)$-stable for all $\eps>0$, and the sequence $(\cA_N)_{N\geq 1}$ is asymptotically $S$-stable if for every $\eps>0$, it is $(S,\eps)$ stable for all $N\geq N(\eps)$ sufficiently large.
\end{definition}

\begin{theorem}
\label{thm:main}
Fix $S>0$ and $E>\ALG(p)$.
Let $(\cA_N)_{N\geq 1}$ be an asymptotically $S$-stable sequence of algorithms.
Then
\[
\lim_{N\to\infty}
\bbP[H_N(\cA_N(H_N))/N\geq E]
=
0.
\]
\end{theorem}

We now explain why Theorem~\ref{thm:main} subsumes low-degree hardness.
Say $\cA^{\circ}_N:\sH_N\times\Omega_N\to\bbR$ is a degree $D$ polynomial if each coordinate of its output is a degree $D$ polynomial (for fixed $\omega$).
We say $\cA_N$ is $C$-regular if $\bbE[\|\cA_N^{\circ}(H_N,\omega)\|^2]\leq CN$.
It follows from e.g.\ \cite[Lemma 3.4 and Theorem 3.1]{gamarnik2020optimization} that $\cA_N^{\circ}$ is $2CD$-stable.
Since $\cA_N^{\circ}$ has no reason to be $\cB_N$-valued, we consider the rounded algorithm 
\[
\cA_N(H_N,\omega)
=
\cA^{\circ}_N(H_N,\omega)\cdot \varphi(\|\cA^{\circ}(H_N,\omega)\|/\sqrt{N}),
\]
where $\varphi(x)=\min(1,1/x)$ for all $x\geq 0$.
In other words, $\cA(H_N,\omega)$ is the closest point to $\cA^{\circ}_N(H_N,\omega)$ within $\cB_N$.
Since the rounding $\cA^{\circ}_N\mapsto \cA_N$ is $1$-Lipschitz, it follows that $\cA_N$ is $2CD$-stable.
Hence the following form of low-degree hardness is a direct consequence of Theorem~\ref{thm:main}.

\begin{corollary}
\label{cor:main}
Fix $D,C,\eps>0$.
For each $N$, let $\cA_N$ be a rounded $C$-regular  degree $D$ polynomial algorithm.
Then
\[
\lim_{N\to\infty}
\bbP[H_N(\cA_N(H_N))/N\geq \ALG(p)+\eps]
=
0.
\]
\end{corollary}

\subsection{Discussion}
\label{subsec:discussion}

Recent work on spin glass landscapes and related problems points towards the following general belief: for random mean-field objective functions $H_N:\bbR^N\to\bbR$ without planted signal, there is an exact algorithmic threshold $\ALG$ which can be achieved in input-linear time using approximate message passing \cite{montanari2021optimization,ams20,chen2023local,huang2024optimization,montanari2024exceptional} and/or Hessian ascent \cite{subag2018following,montanari2023solving,jekel2025potential}.
By contrast, surpassing $\ALG$ seems likely to require exponential running time $\exp(\tilde\Omega(N))$.
While the former results can generally be made constructive, the latter belief cannot be proved rigorously without first solving $P$ vs $NP$.
In lieu of this, a natural goal in establishing computational hardness is to prove failure of restricted classes of algorithms.
One prominent such class consists of \emph{low-degree polynomials}.
The well-established low-degree heuristic \cite{hopkins2018statistical,kunisky2019notes} posits that degree $\tilde O(D)$ polynomials tend to have the same computational power as time $\exp(\tilde O(D))$ algorithms.
In particular, if degree $o(N)$ polynomials are unable to surpass $\ALG$, this would provide evidence that no sub-exponential $\exp(N^{1-\Omega(1)})$-time algorithm can do so.

As mentioned previously, progress towards such algorithmic hardness has come in two forms.
Works such as \cite{gamarnik2019overlap,gamarnik2020optimization,huang2025strong} have shown that $o(N)$-degree polynomials cannot reach energy $\OPT$ (for $p\geq 4$ even).
These works leverage stability via the \emph{ensemble overlap gap property} (e-OGP), gradually deforming $H_N$ and arguing that if the algorithmic outputs change gradually, some solution along the way must be suboptimal.
Separately, \cite{huang2021tight,huang2023algorithmic} show that \emph{overlap-concentrated} algorithms, which include all $o(\sqrt{N})$-Lipschitz algorithms, cannot exceed the threshold $\ALG$.
Overlap concentration means that given a $q$-correlated pair $H_N,H_{N,q}$, the overlap $\la\cA_N(H_N),\cA_N(H_{N,q})\ra/N\in [-1,1]$ between the corresponding outputs concentrates sharply around its expectation, uniformly in $q\in [0,1]$.
The technique is based on the \emph{branching overlap gap property} (BOGP), which constructs a densely branching ultrametric tree of solutions by applying the same algorithm to a suitably correlated family of $p$-spin Hamiltonians.
Many standard high-dimensional optimization algorithms do obey overlap concentration, including gradient descent and Langevin dynamics on dimension-free time-scales.
On the other hand, it is not hard to show that low-degree polynomials encompass Lipschitz algorithms, in the sense that any $D$-Lipschitz algorithm can be approximated by an $O(D^2)$-degree polynomial in $L^2$ (by truncating its Hermite polynomial expansion).

Ideally, one would like to combine the above results to show the threshold $\ALG$ also obstructs low-degree and other stable algorithms, but this appears difficult in general.
Roughly speaking, the branching OGP requires simultaneous control over all pairwise distances among some constellation of solutions, which is too much information for stability to handle (though see the multi-ensemble-OGP from \cite{gamarnik2021partitioning,gamarnik2022algorithms} for an approach in a related setting).
The same type of question is also very natural in closely related optimization problems on sparse random graphs such as MaxCut or MaxSAT.
Here, the analog of Lipschitz algorithms is the class of \emph{local} algorithms, also known as factors of IID \cite{elek2010borel,lyons2011perfect,lyons2017factors,chen2019suboptimality,alaoui2023local,jones2023random}. 
One may conjecture that local algorithms for random sparse instances of e.g.\ MaxCut or MaxSAT are optimal among polynomial time algorithms, and in particular among low-degree polynomials. Unfortunately no general statement to this effect is known. 
By showing that stability suffices for tight algorithmic hardness, Theorem~\ref{thm:main} constitutes the first progress of its kind towards unifying these two classes of algorithms.

\subsubsection*{Discussion of Proof Technique}

Surprisingly, our proof of Theorem~\ref{thm:main} is not based on a direct OGP argument.
Instead, we give a reduction from low-degree to Lipschitz algorithms.
This is similar in spirit to the work \cite{montanari2024equivalence}, which reduced low-degree polynomials to approximate message passing for the problem of rank $1$ matrix estimation, but used convexity and symmetry properties unavailable in our setting that stem from the signal estimation nature of their problem. (See also \cite{brennan2021statistical} which provided a partial equivalence between low-degree polynomials and statistical query algorithms.)


The special property of \emph{pure} spherical spin glasses we use is that all (approximate) stationary points at energy above $\ALG(p)$ are wells: the Hamiltonian is locally strictly concave at such points, except possibly for $O(1)$ outlier eigenvalues.
This allows us to instead consider the problem of finding such a well; this has separately been conjectured to be algorithmically hard in some generality by \cite{behrens2022dis}, who proposed that \emph{``omnipresent marginal stability in glasses is a consequence of computational hardness''}.
This conjecture was confirmed in the form of hardness for $o(N)$-degree polynomials in the Sherrington--Kirkpatrick spin glass by \cite{huang2025strong}, and for Langevin dynamics for spherical spin glasses.
However the spherical case remains open more generally, even for Lipschitz algorithms.
The main technical thrust of this paper is to show that for the task of well-finding, \emph{Lipschitz and stable algorithms are equivalent}.
This equivalence seems to be valid in some generality.
For pure spherical spin glasses, the special property above then lets us convert the Lipschitz hardness result from \cite{huang2021tight} into tight low-degree hardness.

To prove this equivalence between Lipschitz and stable algorithms, we argue that along a slowly varying path of Hamiltonians, any stable algorithm which outputs wells has very limited freedom at each step, essentially parametrized by a constant-dimensional subspace (corresponding to possible outlier Hessian eigenspaces).
Modulo these low-dimensional subspaces, one can thus recover the output of $\cA_N(H_N)$ by starting from $(H_N',\cA_N(H_N'))$ for an independent $H_N'$ and tracing the moving wells along a continuous path of Hamiltonians from $H_N'$ to $H_N$.
By showing that this ``state following'' procedure can be turned into a Lipschitz algorithm, we deduce that stable algorithms cannot outperform Lipschitz ones; see Subsection~\ref{subsec:overview} for a more detailed overview.
We note that homotopy methods of this kind have significant precedent in both spin glass physics and algorithm design more broadly \cite{barrat1997temperature,zdeborova2010generalization,sun2012following,beltran2011fast,burgisser2011problem}.

\subsection{From Optimization to State Following}

The crucial property of pure $p$-spin Hamiltonians we use comes from a close link between $H_N(\bsig)$ and the Riemannian Hessian $\nabla_{\sph}^2 H_N(\bsig)$, defined below.
For each $\bsig \in S_N$, let $\{e_1(\bsig),\ldots,e_N(\bsig)\}$ be an orthonormal basis of $\bbR^N$ with $e_1(\bsig) = \bsig / \sqrt{N}$.
Let $\cT = \{2,\ldots,N\}$.
Let $\nabla_\cT H_N(\bsig) \in \bbR^\cT$ denote the restriction of $\nabla H_N(\bsig) \in \bbR^N$ to the space spanned by $\{e_2(\bsig),\ldots,e_N(\bsig)\}$, and $\nabla^2_{\cT \times \cT} H_N(\bsig) \in \bbR^{\cT \times \cT}$ analogously.
Define the tangential and rescaled radial derivatives:
\begin{equation}
\label{eq:spherical-calculus-def}
    \nabla_\sph H_N(\bsig) = \nabla_{\cT} H_N(\bsig),
    \qquad \partial_\rd H_N(\bsig) = \lt\la e_1(\bsig), \nabla H_N(\bsig)\rt\ra/\sqrt{N}.
\end{equation}
Then the Riemannian Hessian is:
\begin{equation}
\label{eq:riemannian-hessian}
    \nabla^2_{\sph} H_N(\bsig) = \nabla^2_{\cT \times \cT} H_N(\bsig) - \partial_\rd H_N(\bsig) I_{\cT \times \cT}.
\end{equation}
As suggested by \eqref{eq:riemannian-hessian}, it can be shown that the bulk spectrum of $\nabla^2_{\sph} H_N(\bsig)$ is well approximated (uniformly in $\bsig$) by a rescaled Wigner semicircle density, shifted by $\partial_\rd H_N(\bsig)$. 
Because pure $p$-spin Hamiltonians are homogeneous, one deterministically has 
\begin{equation}
\label{eq:radial-deriv-propto}
\partial_\rd H_N(\bsig)=pH_N(\bsig)/N,\quad\forall \bsig\in\cS_N.
\end{equation}
The connection is encapsulated by the next foundational result (we write $\lambda_j$ for the $j$-th largest eigenvalue of a symmetric matrix).

\begin{proposition}[{\cite[Lemma 3]{subag2018following}}]
\label{prop:wells-energy-for-pure}
    For any $\gamma>0$ there are $k=k(\gamma)$ and $\delta=\delta(\gamma)>0$ such that
    \[
    \big|
    \lambda_j\big(\nabla_{\sph}^2 H_N(\bsig)\big)
    -
    2\sqrt{p(p-1)}
    -
    pH_N(\bsig)
    \big|
    \leq
    \gamma/2
    \]
    holds for all $k\leq j\leq \delta N$ and $\bsig\in\cS_N$ simultaneously, with probability $1-e^{-cN}$.
\end{proposition}

Below, we define \emph{wells} of $H_N$ to be approximate stationary points where the Hessian is locally strictly concave, except for a constant number of outlier eigenvalues.
In fact thanks to Proposition~\ref{prop:wells-energy-for-pure}, for the latter property it is equivalent to lower-bound the radial derivative.
In the next definition, one may view $\gamma>0$ as a fairly small constant, and $\delta>0$ as an extremely small constant (still independent of $N$).

\begin{definition}
\label{def:well}
For $\gamma,\delta>0$ we say $\bsig\in\cS_N$ is a \textbf{$(\gamma,\delta)$-well} for the $p$-spin Hamiltonian $H_N$ if 
\[
\|\nabla_{\sph} H_N(\bsig)\|< \delta\sqrt{N},\quad\quad \text{and}\quad\quad \partial_\rd H_N(\bsig) - 2\sqrt{p(p-1)} >\gamma>0
\]
We let $W(\gamma,\delta)=W(\gamma,\delta;H_N)\subseteq\cS_N$ be the set of such points.
\end{definition}

It is immediate from \eqref{eq:radial-deriv-propto} that finding wells is harder than exceeding energy $\ALG$.
In fact these problems are essentially equivalent, because if one can exceed $\ALG$, one can then run gradient ascent to improve the energy until an approximate stationary point is reached.
Further, this latter step preserves stability of any algorithm.
This is encapsulated in the next proposition, whose routine proof is deferred to Section~\ref{sec:deferred-proofs}.

\begin{proposition}
\label{prop:GD-finds-well}
For any $\delta$ and $E>\ALG(p)$ there are $(I,\eta)$ depending on $(\delta,E)$ and $\gamma\geq \Omega(E-\ALG(p))$ independent of $\delta$ such that the following holds.
Suppose $H_N$ satisfies Proposition~\ref{prop:gradients-bounded} and $H_N(\bsig)/N\geq E>\ALG(p)$.
Let $\bsig=\bsig^{(0)},\bsig^{(1)},\dots$ be spherical gradient ascent iterates taking the form
\[
\bsig^{(i+1)}
=
\frac{\bsig^{(i)}+\eta\nabla_{\sph} H_N(\bsig^{(i)}}{\|\bsig^{(i)}+\eta\nabla_{\sph} H_N(\bsig^{(i)}\|}\cdot \sqrt{N}.
\]
Then:
\begin{enumerate}
    \item At least one of $\bsig^{(1)},\dots,\bsig^{(I)}$ lies in $W(\gamma,\delta)$.
    \item If $\bsig=\cA_N(H_N)$ is the output of a asymptotically $S$-stable algorithm, then each $\bsig^{(k)}$ is the output of an asymptotically $S'$-stable algorithm, for $S'=S'(S,I,\eta)$ independent of $N$.
\end{enumerate}
\end{proposition}

The main technical thrust of our argument is the following result, which transforms any low-degree polynomial algorithm which finds wells into a Lipschitz algorithm that does the same.
Since it is known from \cite{huang2021tight,huang2023algorithmic} that Lipschitz algorithms cannot exceed $\ALG(p)$, this together with Proposition~\ref{prop:GD-finds-well} will imply Theorem~\ref{thm:main}.
We emphasize that while Propositions~\ref{prop:wells-energy-for-pure} and \ref{prop:GD-finds-well} are specific to pure $p$-spin models, Theorem~\ref{thm:well-reduction} extends verbatim to mixed $p$-spin models, and should in fact hold quite generically for random smooth landscapes on high-dimensional manifolds.

\begin{theorem}
\label{thm:well-reduction}
Fix $(\gamma,\delta,S,\eta)$, where $\delta\in (0,\delta_0(\gamma))$ is sufficiently small depending on $\gamma$.
Suppose the asymptotically $S$-stable algorithms $(\cA_{N})_{N\geq 1}$ satisfy
\[
\limsup_{N\to\infty}
\bbP[\cA_{N}(H_{N})\in W(\gamma,\delta;H_{N})]\geq \eta>0.
\]
Then for some $L=L(\gamma,\delta,S,\eta)<\infty$, there exists a sequence of $L$-Lipschitz algorithms $\cA^{\Lip}_{N}$ such that 
\[
\limsup_{N\to\infty}
\bbP[\cA^{\Lip}_{N}(H_{N})\in W(\gamma/3,\delta^{1/4};H_{N})]>0.
\]
\end{theorem}

Next we deduce Theorem~\ref{thm:main} on hardness for stable algorithms from Theorem~\ref{thm:well-reduction}.
Hence our main focus in the rest of the paper will be to establish the latter.

\begin{proof}[Proof of Theorem~\ref{thm:main}]
    Suppose for sake of contradiction that for some $E>\ALG(p)$ and asymptotically $S$-stable algorithms $\cA_N$, we have
    \[
    \limsup_{N\to\infty}
    \bbP[H_{N}(\cA_{N}(H_{N}))/N\geq E]>0.
    \]
    Using Proposition~\ref{prop:GD-finds-well} and using the pigeonhole principle to choose $i\in [I]$, there exists an asymptotically $S'$-stable algorithm such that along a further subsequence, we have for some $\gamma,\delta>0$:
    \[
     \limsup_{N\to\infty}
    \bbP[\cA_{N}(H_{N})\in W(\gamma,\delta;H_N)]>0.
    \]
    Then Theorem~\ref{thm:well-reduction} guarantees the existence of $O(1)$-Lipschitz algorithms $\cA_N^{\Lip}$ such that 
    \[
    \limsup_{N\to\infty}
    \bbP[\cA_N^{\Lip}(H_N)\in W(\gamma/3,\delta^{1/4};H_N)]>0,
    \]
    which directly implies 
    \[
    \limsup_{N\to\infty}
    \bbP[H_N(\cA_N^{\Lip}(H_N))/N\geq \ALG(p)+\Omega(\gamma)]>0.
    \]
    This contradicts the Lipschitz hardness stated in Proposition~\ref{prop:BOGP-input}, completing the proof.
\end{proof}

\subsection{Overview of State Following Reduction}
\label{subsec:overview}

Here we outline the state-following construction used to convert a stable well-finding algorithm to a Lipschitz one in Theorem~\ref{thm:well-reduction}.
Let us first suppose that the wells in question have no outlier eigenvalues, which simplifies the proof significantly.
Choose $\eps>0$ small depending on the parameters $(\gamma,\delta,\eta,D)$ in Theorem~\ref{thm:well-reduction} and $K\gg 1/\eps$.
Let $H_N^{(0)},\dots,H_N^{(K)}$ be a discrete-time Ornstein--Uhlenbeck chain of $p$-spin Hamiltonians, which is jointly centered Gaussian with $H_N^{(i)},H_N^{(j)}$ a $(1-\eps)^{|i-j|}$-correlated pair for each $i,j$.
(I.e. conditionally on $H_N^{(0)},\dots, H_N^{(i)}$ the next Hamiltonian $H_N^{(i+1)}$ is $(1-\eps)$-correlated with $H_N^{(i)}$, and conditionally independent of the past given $H_N^{(i)}$.)
Calling the law of this ensemble $\cL_N(p,K,\eps)$, we can then proceed according to the following steps.

\begin{enumerate}
    \item 
    Assuming $\cA_N$ is stable and finds a well with constant probability, one first shows there is a constant probability for $\cA_N(H_N^{(i)})\in W(\gamma,\delta;H_N^{(i)})$ to hold for all $0\leq i\leq K$ simultaneously, and for $\|\cA_N(H_N^{(i)})-\cA_N(H_N^{(i+1)})\|$ to be small for each $i$.
    (See Lemma~\ref{lem:success-and-stability}, which is taken from \cite{huang2025strong}.)
    This event (together with crude high-probability norm estimates on the various disorder tensors) will be denoted $\Sall(K)$.
    \item 
    Since $H_N^{(i)},H_N^{(i+1)}$ are very similar, the event $\Sall(K)$ induces an essentially unique choice of $\cA_N(H_N^{(i+1)})$ from $\cA_N(H_N^{(i)})$ in each step.
    This is because $H_N^{(i)}$ is uniformly concave around the local maximum $\cA_N(H_N^{(i)})$, so the slight perturbation $H_N^{(i+1)}$ of $H_N^{(i)}$ must have a unique local optimum near $\cA_N(H_N^{(i)})$.
    Hence on the event $\Sall(K)$, the output $\cA_N(H_N^{(K)})$ can be recovered by a ``state following'' algorithm which starts from $(H_N^{(0)},\cA_N(H_N^{(0)}))$ and iteratively uses gradient descent to ``follow the moving well''.
    \item 
    This procedure can be turned into a Lipschitz algorithm $\cA^{\Lip}$, because it works even when $H_N^{(0)}$ and $H_N^{(K)}$ are approximately independent of each other.
    More precisely, given an input Hamiltonian $H_N$ to be optimized, we choose $H_N^{(0)}$ to be an IID copy and let 
    \[
    H_N^{(K)}
    =
    (1-\eps)^K H_N^{(0)}+\sqrt{1-(1-\eps)^{2K}}H_N.
    \]
    Then we can sample the intermediate Hamiltonians $H_N^{(1)},\dots,H_N^{(K-1)}$ from their conditional law under $\cL_N(p,K,\eps)$ given the endpoints just constructed.
    Finally we perform state following to find a well $\bsig_*$ of $H_N^{(K)}$.
    For large $K\gg 1/\eps$, the point $\bsig_*$ will also be a well of $H_N\approx H_N^{(K)}$.
    \item 
    The previous algorithm is Lipschitz in $H_N$ in the following sense.
    Since $(H_N^{(0)},\dots,H_N^{(K)},H_N)\in\bbR^{N^k (K+2)}$ is jointly Gaussian, it is the sum of a linear function of $H_N$ and an independent centered Gaussian vector.
    We let the auxiliary random variable $\omega$ consist of this latter vector, together with $\bsig^{(0)}=\cA_N(H_N^{(0)})$.
    We view state following as a function of $(H_N,\omega)$; it is $O(1)$-Lipschitz because it can be simulated by using gradient descent with (small) dimension-free step size and a (large) dimension-free number of iterations.
\end{enumerate}

The presence of Hessian outliers (i.e. $k(\gamma)>1$ in Proposition~\ref{prop:wells-energy-for-pure}) complicates the above strategy.
In particular, the ``moving well'' can slide unpredictably along near-zero eigenspaces.
The main strategy to deal with this is to exhaust all possible sequences of movements along these $O(1)$-dimensional eigenspaces.

Another complication is that in the presence of Hessian outliers, gradient descent no longer suffices for state following; one can at best use restricted local concavity in a well-conditioned subspace.
Due to the explicit dependence on the evolving outlier eigenspace, this yields only a \emph{locally} Lipschitz partial function $\cA_N^{\LocLip}$ that is not always defined.
To construct a \emph{globally} Lipschitz algorithm, we rescale the output of state following toward $0\in\bbR^N$ by a scalar $a^*\in [0,1]$.
We will arrange that $a^*$ continuously measures the degradation of the conditions needed for state following, with $a^*=0$ whenever state following becomes impossible.
Then $\cA_N^{\LocLip}$ is defined whenever $a^*\neq 0$, so this yields a globally Lipschitz extension $\cA_N^{\Lip}$ to which Proposition~\ref{prop:BOGP-input} can be applied.

\section{Technical Preliminaries}


Given an linear subspace $V\subseteq\bbR^N$, we let $P_V$ denote the corresponding orthogonal projection matrix which is the identity on $V$.
We will also denote $P_V^{\perp}=P_{V^{\perp}}$.
For a non-zero vector $\bsig\in\bbR^N$, we let $P_{\bsig}=P_{\bbR \bsig}$ be the projection onto the $1$-dimensional span of $\bsig$.

By default, we metrize Hamiltonians using the un-normalized Euclidean distance on their coefficients, which are denoted $\bg(H_N)$:
\begin{equation}
    \label{eq:hamiltonian-norm}
    \norm{H_N - H_N'} = 
    \norm{\bg(H_N) - \bg(H_N')}_2.
\end{equation}
For a tensor $\bA \in (\bbR^N)^{\otimes k}$, we also use the operator norm
\[
    \tnorm{\bA}_\op =
    \max_{\|\bsig^1\|_2,\ldots,\|\bsig^k\|_2\leq 1}
    |\la \bA, \bsig^1 \otimes \cdots \otimes \bsig^k \ra|\,.
\]
Although we will not directly use it in this paper, an equivalent definition of $H_N$ is the continuous centered Gaussian process on $\bbR^N$ with covariance
\begin{equation}
    \label{eq:def-xi}
\bbE[H_N(\bsig)H_N(\brho)]
=
N
\Big(\frac{\la \bsig,\brho\ra}{N}\Big)^p.
\end{equation}

\subsection{Smoothness Estimates}

The next standard result ensures that $H_N$ is a smooth function.

\begin{proposition}[{\cite[Corollary 59]{arous2020geometry}}]
    \label{prop:gradients-bounded}
    For any $p$, there exists $C_{\ref{prop:gradients-bounded}},c>0$ independent of $N$ such that the following holds for all $N$ sufficiently large.
    Defining the open convex set
    \[
    K_N=\lt\{H_N\in\sH_N~:~\norm{\nabla^k H_N(\bsig)}_{\op} < C_{\ref{prop:gradients-bounded}} N^{1-\frac{k}{2}}~~~~\forall\, 0\leq k\leq p,\norm{\bsig}_2 \le \sqrt{N}\rt\}\subseteq \sH_N,
    \]
    we have $\bbP[H_N \in K_N] \ge 1-e^{-cN}$.
\end{proposition}

\begin{proposition}[{\cite[Proposition A.2]{huang2021arxiv}}]
    \label{prop:gradstable}
    Let $K_N$ be given by Proposition~\ref{prop:gradients-bounded}.
    There exists $C_{\ref{prop:gradstable}}$ independent of $N$ such that for all $H_N, H'_N \in K_N$ and $\bx, \by \in \bbR^N$ with $\norm{\bx}, \norm{\by} \le 2\sqrt{N}$,
    \[
        \norm{\nabla^k H_N(\bx)-\nabla^k H_N'(\by)}_{\op}
        \le 
        C_{\ref{prop:gradstable}}
        N^{\frac{1-k}{2}}
        \lt(\norm{\bx-\by} + \norm{H_N-H_N'}\rt).
    \]
\end{proposition}

It will be convenient to parametrize movement within $\cS_N$ using the Riemannian exponential map, defined as follows.
Fix $\bsig\in\cS_N$ and $v\in\cS_N\cap \bsig^{\perp}$ and suppose $u=\theta v$ for $\theta$ satisfying the (somewhat arbitrary) bound $|\theta|<1$, i.e. $u\in \cB_N^{\circ}\cap \bsig^{\perp}$, for $\cB_N^{\circ}$ the interior of $\cB_N$.
Then the exponential map is defined by
\begin{equation}
\label{eq:exponential-map-def}
T_{\bsig}(u)
=
\cos(\theta)\bsig + \sin(\theta)v.
\end{equation}
It is well known that for any smooth function $H_N:\cS_N\to\bbR$ and $u\in\bsig^{\perp}$,
\begin{equation}
\label{eq:geodesic-derivatives}
    \frac{\de}{\de t} 
    H_N(T_{\bsig}(tu))\big|_{t=0}=\la\nabla_{\sph}H_N(\bsig),u\ra,
    \quad\quad 
    \frac{\de^2}{\de t^2} 
    H_N(T_{\bsig}(tu))\big|_{t=0}=\la \nabla_{\sph}^2 H_N(\bsig),u^{\otimes 2}\ra.
\end{equation}
Furthermore, $T_{\bsig}$ is a diffeomorphism on $u\in\bsig^{\perp}\cap\cB_N^{\circ}$.
We define the following functions on $\cS_N\times \cB_N$:
\begin{align*}
    F_1(\bsig,u)&=T_{\bsig}(P_{\bsig}^{\perp} u),
    \\
    F_2(\bsig,u)&=H_N(F_1(\bsig,u)).
\end{align*}

The next proposition (whose proof is deferred to Section~\ref{sec:deferred-proofs}) combines Proposition~\ref{prop:gradients-bounded} with smoothness properties of the exponential map.
Below, we define derivatives on $\cS_N\times \bbR^N$ using the product manifold structure.

\begin{proposition}
\label{prop:geodesics-smooth}
If $H_N\in K_N$, then for $(\bsig,u)\in \cS_N\times \cB_N^{\circ}$ and $0\leq k\leq 3$ and $i\in \{1,2\}$:
\[
    \|\nabla^k F_i(\bsig,u)\|_{\op}
    \leq 
    O(N^{\frac{1-k}{2}}).
\]
\end{proposition}

\subsection{An Explicit Cover for the Set of Wells}

For any $\bsig\in W(\gamma,\delta;H_N)$, one can find small $\iota\geq \iota_0(\gamma,\delta,\dots)$ which is the location of a two-sided spectral gap, i.e.
\[
\spec(\nabla_{\sph}^2 H_N(\bsig))\cap \pm [\iota,3\iota]=\emptyset
\]
(Here $\pm [a,b]=[-b,-a]\cup [a,b]$ for $0<a<b$.)
While we will always have $\iota\lesssim \gamma$, its precise value must vary with $\bsig$.
For technical reasons, we will cover $W(\gamma,\delta;H_N)$ by a constant number of subsets of $\cS_N$ in which $\iota$ and $d=|\spec(\nabla_{\sph}^2 H_N(\bsig))\cap [-\iota,\iota]|$ are made explicit. 
The Lipschitz state-following algorithm (to be described) will then depend on $(d,\iota)$.
One reason for doing so is that in constructing a globally Lipschitz extension of state following as described above, it is helpful for $d$ and $\iota$ to be held constant.

Given $\bsig\in\cS_N$, we let $U_{\iota}(\bsig;H_N)\subseteq\bsig^{\perp}\subseteq\bbR^N$ be the span of those eigenvectors of $\nabla_{\sph}^2 H_N(\bsig)$ with eigenvalues in $[-\iota,\iota]$.
We will often write $U_{\iota}^{\perp}(\bsig;H_N)$ for $U_{\iota}(\bsig;H_N)^{\perp}$.
Thus by Proposition~\ref{prop:wells-energy-for-pure}, we have $\dim(U_{\iota}(\bsig;H_N))\leq O_{\gamma,\delta}(1)$ when $\bsig\in W(\gamma,\delta;H_N)$.
For $d\in\bbN$ and $0<a<b$, let
\[
W(\gamma,\delta,d,[a,b];H_N)
=
\{\bsig\in W(\gamma,\delta;H_N)\,:\,\dim(U_{a}(\bsig;H_N))=d,~\spec(\nabla_{\sph}^2 H_N(\bsig))\cap \pm [a,b]=\emptyset\}.
\]
For convenience, we often abbreviate:
\[
W(\gamma,\delta,d,\iota;H_N)
=
W(\gamma,\delta,d,[\iota,3\iota];H_N).
\]
Next we show using the pigeonhole principle and Proposition~\ref{prop:wells-energy-for-pure} that $W(\gamma,\delta;H_N)$ is covered by a constant (independent of $N$) number of these sets.

\begin{proposition}
\label{prop:well-types}
For any $\gamma>0$, there exists a finite set of pairs $(d_j,\iota_j)_{j=1}^J$ independent of $N$ such that for small enough $\delta\in (0,\delta_0(\gamma))$ and sufficiently large $N\geq N_0(\gamma,\delta)$:
\begin{equation}
\label{eq:well-types}
\bbP\Big[W(\gamma,\delta;H_N)=\bigcup_{j=1}^J W(\gamma,\delta,d_j,\iota_j;H_N)\Big]
\geq 
1-e^{-cN}.
\end{equation}
\end{proposition}

\begin{proof}
    Choose $(k,\delta_0)$ depending on $\gamma$ as in Proposition~\ref{prop:wells-energy-for-pure}, and let $0<\delta\leq\delta_0$ be arbitrary.
    Let $\iota_i=10^{i-k-5}\gamma$ for $0\leq i\leq k+3$.
    We claim that the event in \eqref{eq:well-types} holds when $(d_j,\iota_j)$ vary over the set $\{0,1,\dots,k\}\times \{\iota_0,\dots,\iota_{k+3}\}$, so long as the event in Proposition~\ref{prop:wells-energy-for-pure} applies.
    Indeed on the latter event, for any $\bsig\in W(\gamma,\delta;H_N)$, there are at most $k$ eigenvalues of $\nabla_{\sph}^2 H_N(\bsig)$ in $[-\gamma/10,\gamma/10]$.
    By the pigeonhole principle, there exists some $0\leq i\leq k+3$ (depending on $\bsig$) such that none of these eigenvalues have absolute value in $[\iota_i,\iota_{i+1})$.
    Then $\bsig\in W(\gamma,\delta,d,\iota_i;H_N)$ for $0\leq d\leq k$ the number of eigenvalues of $\nabla_{\sph}^2 H_N(\bsig)$ in $[-\iota_i,\iota_i]$.
    This establishes the above claim, completing the proof.
\end{proof}

Another reason for the above construction is that the separation parameter $\iota$ governs the stability of the near-zero eigenspace.
We will use the following simple consequences of the Weyl perturbation inequalities and the Davis--Kahan Sine theorem \cite{davis1970sin}, respectively.

\begin{proposition}
\label{prop:davis-kahan}
Let $A,A'\in\bbR^{N\times N}$ be symmetric matrices with $\|A\|_{\op},\|A'\|_{\op}\leq C$.
Suppose that $\spec(A)\cap [-3\iota,3\iota]\subseteq [-\iota,\iota]$ and $A$ has $d$ eigenvalues (with multiplicity) in $[-\iota,\iota]$, and let $V\subseteq\bbR^N$ be the span of the associated eigenvectors.
Suppose $\|A-A'\|_{\op}=\eps$ is small enough depending on $(d,\iota)$.
Then:
\begin{enumerate}[label=(\roman*)]
    \item 
    \label{it:weyl}
    $A'$ has no eigenvalues in $\pm [1.1\iota,2.9\iota]$, and exactly $d$ eigenvalues in $[-1.1\iota,1.1\iota]$.
    \item 
    \label{it:davis-kahan}
    The span $V'$ of the eigenvectors of $A'$ corresponding to eigenvalues in $[-1.1\iota,1.1\iota]$ satisfies
    \[
    \|P_V-P_{V'}\|_{\op}
    \leq 
    O_{C,d,\iota}(\eps).
    \]
    (In particular, the estimate does not depend on the ambient dimension $N$.)
\end{enumerate}
\end{proposition}

We note that one could equally well replace $([-\iota,\iota]$ and $\pm [\iota,3\iota]$ by $(-\infty,\iota]$ and $[\iota,3\iota]$ above.
However we wanted to emphasize that the only problem comes from near-zero eigenvalues, i.e. a well-conditioned Hessian with some negative eigenvalues poses no technical issues for our approach (which might potentially be useful in future work).

\subsection{Success-and-Stability Along an Ensemble}
\label{subsec:jensen-bound}

As mentioned before, we consider a jointly Gaussian sequence $(H_N^{(0)},H_N^{(1)},\dots,H_N^{(K)})$ of $p$-spin Hamiltonians in which $H_N^{(j)}$ and $H_N^{(i)}$ are $(1-\eps)^{|i-j|}$-correlated for each $i,j$.
Fixing a choice $(d,\iota)$ as in the previous subsection and $(\delta,S,\eps)$, define the following events depending on $(H_N^{(0)},\dots,H_N^{(K)},\bsig^{(0)},\dots,\bsig^{(K)})$:
\begin{equation}
\label{eq:S-events}
\begin{aligned}
    \Ssolve(i)
    &\equiv
    \{\bsig^{(i)} \in W(\gamma,\delta,d,\iota;H_N^{(i)})\}
    , \\
    \Sstab(i)
    &\equiv
    \{\|\bsig^{(i)}-\bsig^{(i+1)}\|/\sqrt{N}
    <
    \delta\},
    \\
    \Sbdd(K)
    &\equiv 
    \lt(\bigcap_{i=0}^K \{H_N^{(i)}\in K_N\}\rt)
    \cap 
    \lt(\bigcap_{i=0}^{K-1}
    \{H_N^{(i+1)}-H_N^{(i)}\in \sqrt{2\eps}K_N\}
    \rt),
    \\
    \Sall(K) &= \Big(\bigcap_{i=0}^K \Ssolve(i)\Big) \cap \Big(\bigcap_{i=0}^{K-1} \Sstab(i)\Big)
    \cap 
    \Sbdd(K)
    .
\end{aligned}
\end{equation}
Given an algorithm $\cA_N$, a common choice in the above will be $\bsig^{(i)}=\cA_N(H_N^{(i)},\omega)$ for each $i$.
We let $\Ssolve(i;\cA_N)$, etc be the corresponding events, which now depend only on the Hamiltonians $(H_N^{(0)},\dots,H_N^{(K)})$.
When $\cA_N$ is specified, we let
\[
\psolve=\bbP[\Ssolve(0;\cA_N)],\quad\quad 
\punstable = 1 - \bbP[\Sstab(0;\cA_N)].
\]
Note that $\bbP[\Sbdd(K)]\geq 1-(2K+1)e^{-cN}$ by Proposition~\ref{prop:gradients-bounded} (and because $(H_N^{(i+1)}-H_N^{(i)})/\sqrt{2\eps} \stackrel{d}{=}H_N$ is marginally another $p$-spin Hamiltonian).
Additionally, if $\cA_N$ is $(S,\eps)$-stable then by definition:
\begin{equation}
\label{eq:punstable-bound}
\punstable
\leq 
S\eps/\delta^2.
\end{equation}

The following lemma from \cite{huang2025strong}\footnote{Although $\Ssolve,\Sstab$ are defined differently, the proof is identical modulo the $(2K+1)e^{-cN}$ contribution from $\Sbdd(K)$.}
gives a positive correlation property for the above events, ensuring that $\bbP[\Sall(K;\cA_N)]\geq \Omega(1)$ whenever $\psolve^2>\punstable+\Omega(1)$.
This is key to show that low-degree polynomials fail with high (as opposed to constant) probability.
Below, we denote $x_+=\max(x,0)$.

\begin{lemma}[{\cite[Lemma 3.1]{huang2025strong}}]
\label{lem:success-and-stability}
With notations as above,
    \[
        \bbP[\Sall(K;\cA_N)]
        \geq (\psolve^2-\punstable)_+^{2K}
        -
        (2K+1)e^{-cN}.
    \]
\end{lemma}

We will also need to consider more lenient analogs of the $\Ssolve$ event, which can be expected to hold when $\bsig^{(i)}$ is given by performing state following.
The crucial parameter $\tau\in [1,1.6]$ will continuously parametrize this leniency, with $\tau=1$ recovering the definitions above (the value $1.6$ is of course rather arbitrary).
We do not define $\Sstab^{\tau}$ events because state following is only capable of taking small steps anyway (this also means $\hSall^1(K)$ is more lenient than $\Sall(K)$).
Noting that $1.6^2<3$, we first let
\begin{equation}
\label{eq:W-tau-def}
W^{\tau}(\gamma,\delta,d,\iota;H_N^{(i)})
=
W(\gamma/\tau,\delta^{1/\tau},d,[\tau\iota,3\iota/\tau];H_N^{(i)}).
\end{equation}
The lenient events, again depending on $(H_N^{(0)},\dots,H_N^{(K)},\bsig^{(0)},\dots,\bsig^{(K)})$, are given by:
\begin{equation}
\label{eq:S-events-tau}
\begin{aligned}
    \hSsolve^{\tau}(i)
    &\equiv
    \{\bsig^{(i)} \in W^{\tau}(\gamma,\delta,d,\iota;H_N^{(i)})\}
    , \\
    \hSbdd^{\tau}(K)
    &\equiv 
    \lt(\bigcap_{i=0}^K \{H_N^{(i)}\in \tau K_N\}\rt)
    \cap 
    \lt(\bigcap_{i=0}^{K-1}
    \{H_N^{(i+1)}-H_N^{(i)}\in \tau\sqrt{2\eps}K_N\}
    \rt)
    \\
    \hSall^{\tau}(K) &= 
    \Big(\bigcap_{i=0}^K \hSsolve^{\tau}(i)\Big) 
    \cap 
    \hSbdd^{\tau}(K)
    .
\end{aligned}
\end{equation}
Let $\tau^*_{\times}(H_N^{(0)},\dots,H_N^{(K)},\bsig^{(0)},\dots,\bsig^{(K)})$ be the infimal value of $\tau\in [1,1.6]$ such that $\wh S_{\times}^{\tau}(K)$ holds for $\times\in \{\solve,\bdd,\all\}$, or $\tau^*_{\times}=1.6$ if $\wh S_{\times}^{1.6}(K)$ does not hold.
Of course, $\tau^*_{\all}=\min(\tau^*_{\solve},\tau^*_{\bdd})$.

\begin{proposition}
\label{prop:lenience-lipschitz}
Each $\tau^*_{\times}$ is an $L_{\ref{prop:lenience-lipschitz}}/\sqrt{N}$-Lipschitz function for some $L_{\ref{prop:lenience-lipschitz}}(d,\iota,\gamma,\delta,\eps)$.
\end{proposition}

\begin{proof}
    It suffices to handle $\times\in \{\solve,\bdd\}$.
    Both follow directly from Proposition~\ref{prop:gradstable}.
\end{proof}


\section{Proof of Theorem~\ref{thm:well-reduction} via State-Following}

In this section, we explain how to turn any stable well-finding algorithm into a Lipschitz algorithm via state following.
We begin in Subsection~\ref{subsec:1-step} by examining a single step of state following, and carefully handling the presence of outlier eigenvalues.
We then combine these steps to obtain a state-following algorithm which is \emph{locally} Lipschitz at locations where $\Sall(1)$ holds.
Finally we extend this algorithm to be \emph{globally} Lipschitz, thus establishing Theorem~\ref{thm:well-reduction}.

Before getting into the detailed arguments, let us specify the relative sizes of the various parameters that will be used in the arguments below.
Given $E>\ALG(p)$ we first take $\gamma=\Theta(E-\ALG(p))>0$.
We then choose $\delta$ sufficiently small given $\gamma$, as specified in Theorem~\ref{thm:well-reduction} (and also small enough for Proposition~\ref{prop:well-types}).
As Theorem~\ref{thm:well-reduction} does not constrain $(S,\eta)$, we may assume without loss of generality (i.e. pessimistically) they are respectively very large and small depending on $(\gamma,\delta)$.
Using the pigeon-hole principle, we then select $(d,\iota)=(d_j,\iota_j)$ depending on $\gamma$ as in Proposition~\ref{prop:well-types}, such that
\[
\eta'
=
\limsup_{N\to\infty}
\bbP[\cA_{N}(H_{N})\in W(\gamma,\delta,d,\iota;H_{N})]
>
\eta/J.
\]
We then choose $\eps$ small, and finally $K$ large.
This can be written informally as 
\begin{equation}
\label{eq:parameter-order}
\gamma=\Theta(E-\ALG(p))
\gg
d^{-1},\iota
\gg\delta
\gg S^{-1},\eta
\gg \eps
\gg 1/K
>0.
\end{equation}
(Note that although $(d,\iota)$ is \emph{selected} based on $\eta$, it is contained within a finite set depending only on $\gamma$.)

\subsection{One Step of State Following}
\label{subsec:1-step}

In this subsection, we fix $(d,\iota)$ and consider a single pair of $1-\eps$ correlated Hamiltonians $(H_N,\wt H_N)$, given by
\begin{equation}
\label{eq:HN-wtHN}
\wt H_N=(1-\eps)H_N+\sqrt{1-(1-\eps)^2}H_N'
\end{equation}
for $(H_N,H_N')$ an IID pair of $p$-spin Hamiltonians. 
The lemmas below detail the behavior of a single state following step for this correlated pair.
Their proofs are presented in Section~\ref{sec:deferred-proofs}.

These lemmas will be applied later in this section with $(H_N,\wt H_N)=(H_N^{(i)},H_N^{(i+1)})$ for each $0\leq i\leq K-1$.
We accordingly write $\hSsolve^{\tau}(H_N,\bsig)$ and $\hSbdd^{\tau}(H_N,\wt H_N)$ and $\hSall^{\tau}(H_N,\wt H_N;\cA_N)$ to denote the corresponding events from Subsection~\ref{subsec:jensen-bound} with $(H_N,\bsig)$ in place of $(H_N,\cA_N(H_N))$ and $(H_N,\wt H_N)$ in place of $(H_N^{(0)},H_N^{(1)})$.

Below, we choose a sufficiently large implicit constant $C(d,\iota)$ and set
\[
\delta'=C(d,\iota)\delta,
\quad\quad
\delta''=\delta^{0.6}\gg \delta^{1/1.6}.
\]
The next lemma defines the main step of state following, and shows that given $(\cA(H_N),H_N)$, it is essentially parametrized by a low-dimensional vector $u\in U_{\iota}(\cA(H_N);H_N)$.

\begin{lemma}
\label{lem:state-following-1-step}
Suppose $\hSbdd^{1.6}(H_N,\wt H_N)$ and $\hSsolve^{1.6}(H_N,\bsig)$ hold, and let $\bsig\in W^{1.6}(\gamma,\delta,d,\iota;H_N)$.
For each vector $u\in U_{\iota}(\bsig;H_N)$ with $\|u\|< \delta''\sqrt{N}$, there exists a unique $\wt\bsig(u)\in\cS_N$ such that $y=T_{\bsig}^{-1}(\wt\bsig(u))$ satisfies:
\begin{align}
y-u&\in 
U_{\iota}^{\perp}(\bsig;H_N),
\\
\|y-u\|&\leq O_{d,\iota}(\delta\sqrt{N}),
\\
\label{eq:restricted-stationary-point}
[\nabla (\wt H_N\circ T_{\bsig})](y)
&\in U_{\iota}(\bsig(u);H_N).
\end{align}
\end{lemma}

Given $F:U\to\bbR^m$ for $U\subseteq\bbR^M$ open, we say $F$ is locally $L$-Lipschitz at $x$ if there is a neighborhood of $x$ on which $F$ is $L$-Lipschitz.
If this holds for all $x\in U$, we say $F$ is locally $L$-Lipschitz.
The next lemma shows that state following is locally Lipschitz.
Then Lemma~\ref{lem:state-following-adjust-basis} shows a new orthonormal basis around the resulting point can also be found in a locally Lipschitz way (which is important for iterating).

\begin{lemma}
\label{lem:state-following-local-lipschitz-1-step}
The map $(H_N,\wt H_N,\bsig,u)\mapsto \wt\bsig(u)$ is locally $L_{\ref{lem:state-following-local-lipschitz-1-step}}(d,\iota)$ Lipschitz at any $(H_N,\wt H_N,\bsig,u)$ for which the conditions of Lemma~\ref{lem:state-following-1-step} apply.
(In particular such $(H_N,\wt H_N,\bsig,u)$ form an open set.)
\end{lemma}

\begin{lemma}
\label{lem:state-following-adjust-basis}
Suppose the conditions of Lemma~\ref{lem:state-following-1-step} apply and $(v_1,\dots,v_d)$ is an orthonormal basis for $U_{\iota}(\bsig;H_N)$.
Then there exists a locally $L_{\ref{lem:state-following-adjust-basis}}(d,\iota)$-Lipschitz map of $(H_N,\wt H_N,u,v_1,\dots,v_d)$ that outputs an orthonormal basis $(w_1,\dots,w_d)$ of $U_{\iota}(\wt\bsig(u);\wt H_N)$.
\end{lemma}

Finally when $\Sall(H_N,\wt H_N;\cA_N)$ holds, we show the gradient norm remains small for suitable $u$.
This will let us close the induction to perform state following for $K$ steps in the next subsection, by ensuring that the gradient norm does not grow over time.

\begin{lemma}
\label{lem:state-following-works-1-step}
In the setting of Lemma~\ref{lem:state-following-1-step}, suppose that $\Sall(H_N,\wt H_N;\cA_N)$ holds and that $u\in U_{\iota}(\bsig;H_N)$ satisfies
\begin{equation}
\label{eq:correct-u-for-next-step}
\|u-P_{U_{\iota}(\bsig;H_N)} T_{\bsig}^{-1}(\cA(\wt H_N))\|\leq \delta^2\sqrt{N}.
\end{equation}
Then we have:
    \begin{equation}
    \label{eq:state-following-works}
    \begin{aligned}
    \|\wt\bsig(u)-\cA_N(\wt H_N)\|&\leq \delta'\sqrt{N},
    \\
    \|\nabla_{\sph} \wt H_N(\wt \bsig(u))\|&\leq \delta'\sqrt{N}.
    \end{aligned}
    \end{equation}
\end{lemma}



\subsection{A Locally Lipschitz State Following Algorithm}

We continue to fix $(d,\iota)$.
Here we explain how to perform the $1$-step state following from the previous subsection $K$ times.
By doing so, we will construct a \emph{locally} $O(1)$-Lipschitz partial function $\cA_N^{\LocLip}$, which is well-defined whenever $\hSall^{1.6}(K)$ holds along the state following trajectory.
$\cA_N^{\LocLip}$ depends on the following auxiliary parameters, which together constitute the independent randomness $\omega$.
In fact $\cA_N^{\LocLip}(H_N,\omega)$ will be locally Lipschitz in both inputs, although only the former is directly relevant for Proposition~\ref{prop:BOGP-input}.
\begin{itemize}
    \item 
    A $p$-spin Hamiltonian $H_N^{(0)}$, together with $\bsig^{(0)}=\cA_N(H_N^{(0)},\omega^*)$.
    (Here $\omega^*$ is the independent randomness used in the original stable algorithm $\cA_N$.)
    \item An orthonormal basis $(v_1^{(0)},\dots,v_d^{(0)})$ for $U_{\iota}(\bsig^{(0)};H_N^{(0)})$, assuming that $\Ssolve(0)$ holds (with the specified choice of $(d,\iota)$). 
    (If $\Ssolve(0)$ does not hold, then $(v_1^{(0)},\dots,v_d^{(0)})$ is arbitrary; additionally $(v_1^{(0)},\dots,v_d^{(0)})$ must depend Borel measurably on $H_N^{(0)}$.)
    \item 
    Vectors $\wt u^{(0)},\dots,\wt u^{(K-1)}\in\bbR^d$, which are IID uniform in the set $\{\wt u\in\bbR^d:\|\wt u\|<\sqrt{N}\}$.
    \item IID standard Gaussian $p$-tensors $\bG^{(1)},\dots,\bG^{(K)}\in\bbR^{N^p}$.
\end{itemize}

Given an input disorder $H_N$ and the above data, we first construct a sequence $(H_N^{(0)},\dots,H_N^{(K)})$ with law $\cL_N(p,K,\eps)$, i.e.\ so that $H_N^{(i)},H_N^{(j)}$ are $(1-\eps)^{|i-j|}$-correlated.
To do so, let
\[
H_N^{(K)}
=
(1-\eps)^K H_N^{(0)}+ \sqrt{1-(1-\eps)^{2K}}H_N
\]
so that $H_N^{(0)},H_N^{(K)}$ are $(1-\eps)^K$ correlated.
We then iteratively for $k=1,2,\dots,K-1$ set
\[
H_N^{(k)}=\bbE^{\cL_N(p,K,\eps)}[H_N^{(k)}|H_N^{(k-1)},H_N^{(K)}]+a(K,k)\bG^{(k)}
\]
for suitable constants $a(K,k)\in [0,1]$.
It is not hard to see that there exist constants $a(K,k)$ such that the above is equivalent to sequential sampling from $\cL_N(p,K,\eps)$, so the resulting ensemble has law $\cL_N(p,K,\eps)$.

We will define $\cA_N^{\LocLip}$ by $K$ iterations of state-following, via a sequence $\bsig^{(1)},\dots,\bsig^{(K)}\in\cS_N$ of approximate algorithmic outputs with $\cA_N^{\LocLip}(H_N,\omega)=\bsig^{(K)}$.
Simultaneously, we will construct for each $1\leq j\leq K$ an orthonormal basis $(v^{(j,1)},\dots,v^{(j,d)})$ for $U_{\iota,j}\equiv U_{\iota}(\bsig^{(j)};H_N^{(j)})$.
In each step, given $\bsig^{(j)}$ and the orthonormal basis $(v^{(j,1)},\dots,v^{(j,d)})$ for $U_{\iota,j}$ the next objects for $j+1$ are given as follows:
\begin{enumerate}[label=(\roman*)]
    \item 
    \label{it:next-translate}
    Let $u^{(j)}=\sum_{i=1}^d \wt u^{(j)}_i v^{(j,i)}\in U_{\iota,j}$.\footnote{In words, the auxiliary $\wt u^{(j)}$ encodes the coordinates of the translation vector $u^{(j)}\in U_{\iota,j}$ in the current basis $(v^{(j,1)},\dots,v^{(j,d)})$.}
    \item 
    \label{it:next-iterate}
    Use Lemma~\ref{lem:state-following-1-step} with $(H_N,\wt H_N,\bsig,u)=(H_N^{(j)},H_N^{(j+1)},\bsig^{(j)},u^{(j)})$ to construct $\bsig^{(j+1)}=\wt \bsig(u)$.
    \item 
    \label{it:next-basis}
    Use Lemma~\ref{lem:state-following-adjust-basis} to compute the next basis $(v^{(j+1,1)},\dots,v^{(j+1),d})=(w_1,\dots,w_d)$ for $U_{\iota,j+1}$.
\end{enumerate}

Assuming the conditions of Lemmas~\ref{lem:state-following-1-step} and Lemma~\ref{lem:state-following-adjust-basis} hold at each stage, we let 
\[
\cA_N^{\LocLip}=\bsig^{(K)}
\]
be the final output of the above iteration.
If either condition is false at any time, then the output of $\cA_N^{\LocLip}$ is not defined.
Altogether, these conditions amount to requiring that $\hSall^{1.6}(K)$ holds (with the implicit understanding that if $\bsig^{(k)}$ causes $\hSall^{1.6}(k)$ to be violated, then future iterates $\bsig^{(k+1)},\dots$ are not actually defined).
We next show that these conditions are indeed satisfied for some specific choices of $\wt u^{(0)},\dots,\wt u^{(K-1)}$ as long as $\Sall(K;\cA_N)$ holds.
Here it is crucial that the lemmas from the previous subsection use tolerance parameter $\tau>1$, because $\Sall(K;\cA_N)$ will guarantee only that $\bsig^{(j)}\approx \cA_N(H_N^{(j)};\omega)$ rather than exact equality.

\begin{lemma}
\label{lem:multi-step-state-following-makes-sense}
Suppose $\Sall(K;\cA_N)$ holds for $(H_N^{(0)},\dots,H_N^{(K)})$ above and $H_N\in K_N$. 
Then there exist $\wt u^{(0)},\dots,\wt u^{(K-1)}$ such that:
\begin{enumerate}[label=(\alph*)]
    \item 
    \label{it:lemmas-apply}
    $\hSall^{1.2}(K)$ holds.
    In particular, the conditions of Lemmas~\ref{lem:state-following-1-step} and \ref{lem:state-following-adjust-basis} hold in each step of \ref{it:next-iterate}, \ref{it:next-basis}, and so $\cA_N^{\LocLip}=\bsig^{(K)}$ is well-defined.
    \item 
    \label{it:good-approximation}
    We have $\|\bsig^{(k)}-\cA_N(H_N^{(k)})\|\leq \delta^{0.9}\sqrt{N}$ for each $0\leq k\leq K$.
    \item 
    \label{it:state-following-succeeds}
    The final output satisfies $\bsig^{(K)}\in W(\gamma/2,\delta^{1/3};H_N)$.
\end{enumerate}
\end{lemma}

\begin{proof}
For the first two parts, we induct together on $k$; the inductive hypothesis for \ref{it:lemmas-apply} is $\hSall^{1.2}(k)$.
Since $\Sall(K;\cA_N)$ includes $\Sbdd(K;\cA_N)=\hSbdd^1(K)$, it suffices to consider the $\hSsolve^{1.2}(k)=\hSsolve^{1.2}(H_N^{(k)},\bsig^{(k)})$ events in the induction.
The base case $k=0$ holds by definition as $\bsig^{(0)}=\cA_N(H_N^{(0)})\in W(\gamma,\delta;H_N^{(0)})$.

We first perform the inductive step for \ref{it:good-approximation}.
This follows from Lemma~\ref{lem:state-following-works-1-step}.
Firstly, $\Sall(H_N^{(k)},H_N^{(k+1});\cA_N)$ holds by assumption. 
Second, we may choose an appropriate $u\in U_{\iota,k+1}$ satisfying \eqref{eq:correct-u-for-next-step} and identifying the corresponding $\wt u^{(k+1)}$ under the linear isometry from Lemma~\ref{lem:state-following-adjust-basis}.
(Note that $\delta'=O_{d,\iota}(\delta)\leq o(\delta^{0.9})$ which closes the induction.)

For the inductive step on \ref{it:lemmas-apply}, we combine the just-proved part~\ref{it:good-approximation} with the last part of Lemma~\ref{lem:state-following-works-1-step}.
Checking the conditions of $\hSsolve^{1.2}(H_N^{(k+1)},\bsig^{(k+1)})$ follows similarly to Proposition~\ref{prop:lenience-lipschitz}.
This completes the inductive step and hence establishes \ref{it:lemmas-apply} and \ref{it:good-approximation}.

Finally for \ref{it:state-following-succeeds}, as $K$ is large depending on $(\eps,\delta)$, the assumption of $\Sall(K;\cA_N)$ implies the following.
If $x,y\in\cS_N$ satisfy $x\in W(\gamma,\delta;H_N^{(K)})$ and $\|y-x\|\leq \delta^{0.9}\sqrt{ N}$, then $y\in W(\gamma/2,\delta^{1/3};H_N)$.
Indeed, we chose $K$ large depending on $\eps$ and thus have $H_N-H_N^{(K)}\in o_{K\to\infty}(1)\cdot K_N$. Together with Proposition~\ref{prop:gradstable}, this yields the claim and thus \ref{it:state-following-succeeds} by taking $(x,y)=(\cA_N(H_N^{(K)}),\bsig^{(K)})$.
\end{proof}

\begin{lemma}
\label{lem:loc-lip-property}
$\cA_N^{\LocLip}$ is locally $L_{\ref{lem:loc-lip-property}}(d,\iota,\delta,S,\eta,\eps,K)$-Lipschitz in $(H_N,\omega)$ on its domain, namely the set on which $\hSall^{1.6}(K)$ holds for $(H_N^{(0)},\dots,H_N^{(K)},\bsig^{(0)},\dots,\bsig^{(K)})$ defined above.
\end{lemma}

\begin{proof}
    Using Lemmas~\ref{lem:state-following-local-lipschitz-1-step} and \ref{lem:state-following-adjust-basis}, it can be shown by a direct induction on $k$ that each $\bsig^{(k)}$ is locally Lipschitz on the domain as claimed.
\end{proof}

We now conclude this subsection by showing that $\cA_N^{\LocLip}$ finds wells with positive probability under the assumptions of Theorem~\ref{thm:well-reduction} (but for fixed $(d,\iota)$).
The only remaining task for the next subsection will be to extend the restriction $\cA_N^{\LocLip}|_{\hSall^{1.6}(K)}$ to a globally Lipschitz function.

\begin{corollary}
\label{cor:random-u-pos-prob}
Fix $(\gamma,d,\iota,\delta,S,\eta)$ as in \eqref{eq:parameter-order}.
Suppose the asymptotically $S$-stable algorithms $(\cA_{N})_{N\geq 1}$ satisfy
\[
\limsup_{N\to\infty}
\bbP[\cA_{N}(H_{N})\in W(\gamma,\delta,d,\iota;H_{N})]\geq \eta>0.
\]
For $\alpha=\alpha(\gamma,d,\iota,\delta,S,\eta,\eps,K)>0$,
\[
\limsup_{N\to\infty}
\bbP\Big[\hSall^{1.3}(K)\cap
\{\cA_N^{\LocLip}(H_N,\omega)\in W(\gamma/3,\delta^{1/4};H_N)\}\Big]\geq \alpha.
\]
\end{corollary}

\begin{proof}
First, for parameters as in \eqref{eq:parameter-order}, the estimate \eqref{eq:punstable-bound} implies via Lemma~\ref{lem:success-and-stability} that 
\[
\limsup_{N\to\infty}
\bbP[\Sall(K;\cA_N)]\geq \alpha_0(\gamma,\delta,S,\eta,\eps,K)>0.
\]
On this event, Lemma~\ref{lem:multi-step-state-following-makes-sense}\ref{it:state-following-succeeds} implies there exist $\wt u^{(1)},\dots,\wt u^{(K)}$ with $\cA_N^{\LocLip}(H_N,\omega)\in W(\gamma/2,\delta^{1/3};H_N)$.
By Lemma~\ref{lem:loc-lip-property}, to have $\cA_N^{\LocLip}(H_N,\omega)\in W(\gamma/3,\delta^{1/4};H_N)$ it suffices for each $\wt u^{(1)},\dots,\wt u^{(K)}$ to be within $\rho\sqrt{N}$ from the choices identified in Lemma~\ref{lem:multi-step-state-following-makes-sense}, for some $\rho(\gamma,d,\iota,\delta,S,\eta,\eps,K)>0$.
(In particular Lemmas~\ref{lem:multi-step-state-following-makes-sense}\ref{it:lemmas-apply} and \ref{lem:loc-lip-property} easily imply that such $(\wt u^{(1)},\dots,\wt u^{(K)})$ satisfy $\hSall^{1.3}(K)$, since as in the proof of Proposition~\ref{prop:lenience-lipschitz}, each $\tau^*_{\times}$ is Lipschitz in all relevant quantities.)
These $\rho\sqrt{N}$ approximations hold with uniformly positive probability (i.e. not tending to zero with $N$) as each $\wt u^{(k)}$ is uniformly random in a $d$-dimensional ball of radius $\sqrt{N}$.
\end{proof}

\subsection{Finishing the Proof: Globally Lipschitz State Following}

We start with an elementary fact.

\begin{proposition}
\label{prop:locally-Lipschitz-global}
If $F:\bbR^M\to\bbR^m$ is locally $L$-Lipschitz at all $x\in\bbR^M$,
then $F$ is globally $L$-Lipschitz.
\end{proposition}

\begin{proof}
    For any $x,y\in\bbR^M$ the line segment connecting them is compact. 
    Hence there exists a finite collection of open sets $U_1,\dots,U_J\subseteq\bbR^M$ covering this line segment, such that $F$ is $L$-Lipschitz on each.
    This easily shows that $\|F(x)-F(y)\|\leq L\|x-y\|$.
    Since $x,y$ were arbitrary this completes the proof.
\end{proof}

\begin{proposition}
\label{prop:rescaled-global-lipschitz}
There exists $\cA_N^{\Lip}:\sH_N\times\Omega_N\to \cB_N$ such that:
\begin{itemize}
    \item If $\hSall^{1.3}(K)$ holds for $(H_N^{(0)},\dots,H_N^{(K)},\bsig^{(0)},\dots,\bsig^{(K)})$as constructed in the previous subsection, then $\cA_N^{\Lip}(H_N,\omega)=\cA_N^{\LocLip}(H_N,\omega)\in\cS_N$.
    \item $\cA_N^{\Lip}$ is globally $L=L(\gamma,d,\iota,\delta,S,\eta,\eps,K)$-Lipschitz. 
\end{itemize} 
\end{proposition}

\begin{proof}
We define $\cA_N^{\Lip}$ to be a radial rescaling of $\cA_N^{\LocLip}$.
With $\tau^*=\tau^*_{\all}(H_N^{(0)},\dots,H_N^{(K)},\bsig^{(0)},\dots,\bsig^{(K)})$ as in
Proposition~\ref{prop:lenience-lipschitz} and $a^*=\max(0,\min(1,14-10\tau^*))$, we let:
\[
\cA_N^{\Lip}(H_N,\omega)
=
a^*\cA_N^{\LocLip}(H_N,\omega)\in\cB_N.
\]
Note that although $\cA_N^{\LocLip}$ is not defined on all inputs, whenever it is not defined we have $\tau^*=1.6$ so $a^*=0$ and then $\cA_N^{\Lip}(H_N,\omega)=0\in\cB_N$.
With this notational understanding, $\cA_N^{\Lip}$ is thus a globally defined algorithm.
The first claim holds by definition since $14-10\cdot 1.3=1$.

Turning to the Lipschitz constant, we will show $\cA_N^{\Lip}$ is locally $O(1)$-Lipschitz in $H_N$ at every input $(H_N,\omega)$, which suffices by Proposition~\ref{prop:locally-Lipschitz-global}.
First suppose $\tau^*\geq 1.5$.
Then $a^*=0$ and so $\cA_N^{\Lip}(\wh H_N,\omega)=0\in\cB_N$ for all $\wh H_N$ in a neighborhood of $H_N$, so clearly $\cA_N^{\Lip}$ is locally Lipschitz.

Next suppose $\tau^*\leq 1.5$.
Then $\cA_N^{\LocLip}$ is $\sqrt{N}$-bounded and locally $O(1)$-Lipschitz by Lemma~\ref{lem:loc-lip-property}.
Meanwhile $\tau^*$ is $O(1)$-bounded, and is locally $O(N^{-1/2})$ Lipschitz by Proposition~\ref{prop:lenience-lipschitz} and the local Lipschitz dependencies in Lemmas~\ref{lem:state-following-local-lipschitz-1-step} and \ref{lem:state-following-works-1-step} and \ref{lem:loc-lip-property}.
Since the product of respectively $B_1,B_2$-bounded and $L_1,L_2$-Lipschitz functions has Lipschitz constant at most $B_1 L_2 + B_2 L_1$, the second claim follows.
\end{proof}

\begin{proof}[Proof of Theorem~\ref{thm:well-reduction}]
By the pigeonhole principle, there exists $1\leq j\leq J$ as in Proposition~\ref{prop:well-types} such that 
\[
\limsup_{N\to\infty}
\bbP[\cA_N(H_N)\in W(\gamma,\delta,d_j,\iota_j;H_N)]\geq \frac{\eta}{2J}>0.
\]
Using $\eta/2J$ in place of $\eta$ in \eqref{eq:parameter-order} and combining Corollary~\ref{cor:random-u-pos-prob} with Proposition~\ref{prop:rescaled-global-lipschitz}, we find that 
\[
\limsup_{N\to\infty}\bbP[\cA_N^{\Lip}(H_N)\in W(\gamma/3,\delta^{1/4};H_N)]
>0.
\]
This contradicts Proposition~\ref{prop:BOGP-input}, completing the proof.
\end{proof}


\section{Deferred Proofs}
\label{sec:deferred-proofs}

\begin{proof}[Proof of Proposition~\ref{prop:GD-finds-well}]
Choosing $\eta$ small compared to $C_{\ref{prop:gradstable}}$ and using Proposition~\ref{prop:gradients-bounded}, one finds:
\begin{align*}
\frac{H_N(\bsig^{(i+1)})}{N}
&\geq 
\frac{H_N(\bsig^{(i)})}{N}
+
\frac{\eta\|\nabla_{\sph}H_N(\bsig^{(i)})\|^2}{N}
-
O\lt(\frac{C_{\ref{prop:gradstable}}\eta^2\|\nabla_{\sph}H_N(\bsig^{(i)})\|^2}{N}\rt)
\\
&\geq 
\frac{H_N(\bsig^{(i)})}{N}
+
\frac{\eta\|\nabla_{\sph}H_N(\bsig^{(i)})\|^2}{2N}.
\end{align*}
Choose $I=\lceil 10C_{\ref{prop:gradstable}}\delta^{-2}\eta^{-1}\rceil$ and let $1\leq i\leq I$ be the first time that 
$\|\nabla_{\sph}H_N(\bsig^{(i)})\|\leq \delta$.
Such $i$ exists because if not, we would have $H_N(\bsig^{(I)})>C_{\ref{prop:gradstable}}\geq \max_{\bsig\in\cS_N} \frac{H_N(\bsig)}{N}$ (from the $i=0$ case of Proposition~\ref{prop:gradstable}).
Then 
\[
H_N(\bsig^{(i)})/N\geq H_N(\bsig^{(0)})/N\geq \ALG+\eps.
\]
Using Proposition~\ref{prop:wells-energy-for-pure}, it follows that $\bsig^{(i)}\in W(\gamma,\delta;H_N)$ for $\gamma\geq \Omega(E-\ALG(p))>0$.

The second assertion is a simple discrete Gronwall argument using Proposition~\ref{prop:gradients-bounded}.
Namely, suppose $H_N,\wt H_N\in K_N$ are $(1-\eps)$-correlated as in \eqref{eq:HN-wtHN}.
Letting $\wt\bsig^{(0)}=\cA_N(\wt H_N)$, we have for $N$ large:
\[
\bbE[\|\bsig^{(0)}-\wt \bsig^{(0)}\|^2]\leq S\eps N.
\]
Let $R^{(j)}=\|\bsig^{(j)}-\wt\bsig^{(j)}\|/\sqrt{N}$ for $0\leq j\leq K$.
Assuming $H_N,\wt H_N,H_N'\in K_N$, Proposition~\ref{prop:gradstable} yields
\[
\|\nabla H_N(\bsig)-\nabla \wt H_N(\wt\bsig)\|
\leq 
O(\|H_N-\wt H_N\|+\|\bsig-\wt\bsig\|)
\leq 
O(\sqrt{\eps N}+\|\bsig-\wt\bsig\|).
\]
for all $\bsig,\wt\bsig\in\cS_N$. 
Observing that $\|\bsig^{(i)}+\eta\nabla_{\sph} H_N(\bsig^{(i)}\|\geq \sqrt{N}$ at each stage and that $\bsig\mapsto\frac{\bsig\sqrt{N}}{\|\bsig\|}$ is a contraction on the complement of $\cB_N$, one obtains the recursion
\[
R^{(i+1)}
\leq 
R^{(i)}\cdot (1+O(\eta))+ O(\sqrt{\eps}).
\]
Therefore $\max_{1\leq i\leq I} R^{(j)}\leq (1+O(\eta))^{I} (R^{(0)}+\sqrt{\eps})$. 
Since the event $H_N,\wt H_N,H_N'\in K_N$ has probability $1-e^{-cN}$, we conclude that 
\[
\max_{1\leq i\leq I}
\bbE[(R^{(i)})^2]
\leq 
(1+O(\eta))^{2I} S\eps 
+
e^{-\Omega(N)}.
\]
This shows each $\bsig^{(i)}$ is asymptotically $S'(S,I,\eta)$-stable (since the $e^{-\Omega(N)}$ term becomes negligible for any fixed $\eps>0$ and large $N$).
\end{proof}

\begin{proof}[Proof of Proposition~\ref{prop:geodesics-smooth}]
    We recall that the Riemannian exponential map $T_{\bsig}(u)$ is $C^{\infty}$ jointly in $(\bsig,u)$, see e.g. \cite[Proof of Lemma 5.12]{lee2006riemannian}. This implies that $F_1$ is infinitely differentiable since projections are smooth.
    To quantitatively estimate the derivatives of $F_1$, we simply note that the $k$-th derivative of $F_1$ is defined by $\bsig,u$ and vectors $v_1,\dots,v_k\in \bsig^{\perp}$ and $w_1,\dots,w_k\in\bbR^N$, tangent to $\bsig$ and $u$ respectively.
    By restricting to the linear span of these vectors, we see that the $k$-th derivative tensor of $F_1$ can be expressed using only the geometry of an $O(1)$-dimensional sphere.
    Hence there is no dimension-dependence if one rescales this sphere to have diameter $1$ rather than $\sqrt{N}$ (since the computation becomes entirely independent of $N$).
    This yields the smoothness estimates on $F_1$.
    The smoothness estimates on $F_2$ follow by repeated differentiation using the definition of $K_N$.
\end{proof}

\begin{proof}[Proof of Lemma~\ref{lem:state-following-1-step}]
First we argue existence.
Consider the function 
\begin{equation}
\label{eq:F-def}
F(y)=\wt H_N(T_{\bsig}(y)), \quad \forall \,y\in\bsig^{\perp}.
\end{equation}
Then writing $P=P_{U_{\iota}(\bsig;H_N)}$ and $P^{\perp}=P_{U_{\iota}(\bsig;H_N)}^{\perp}$, we have from \eqref{eq:geodesic-derivatives}:
\begin{align}
\nabla F(0)= 
\nabla_{\sph} \wt H_N (\bsig),
\quad\quad 
\nabla^2 F(0)= 
\nabla_{\sph}^2 \wt H_N(\bsig) 
.
\end{align}
For any $u\in U_{\iota}(\bsig;H_N)$ with $\|u\|\leq\delta''\sqrt{N}$, we find using Proposition~\ref{prop:gradients-bounded} to control Taylor expansion errors:
\begin{align*}
    \nabla F(u)
    =
    \nabla_{\sph} \wt H_N (\bsig)
    + 
    \nabla_{\sph}^2 \wt H_N(\bsig) 
    [u]
    +
    O((\delta'')^2 \sqrt{N})
    .
\end{align*}
(Here the last term is a vector with norm $O((\delta'')^2 \sqrt{N})$.)
Since $U_{\iota}(\bsig;H_N)$ is an eigenspace of $\nabla_{\sph}^2 H_N(\bsig)$,
\[
P^{\perp}\nabla_{\sph}^2 H_N(\bsig)u=0.
\]
Further, $\Sbdd^{1.6}(H_N,\wt H_N)$ implies 
\[
\|\nabla_{\sph}^2 H_N(\bsig)-\nabla_{\sph}^2 \wt H_N(\bsig)\|_{\op}\leq O(\sqrt{\eps}).
\]
Hence as $\eps$ is small compared to $\delta''$, we conclude
\begin{equation}
\label{eq:u-approx-stationary}
\|P^{\perp}\nabla F(u)\|
\leq
\|\nabla_{\sph}^2 H_N(\bsig)-\nabla_{\sph}^2 \wt H_N(\bsig)\|_{\op}\sqrt{N}
+
O((\delta'')^2 \sqrt{N})
\leq 
O(\delta \sqrt{N})
.
\end{equation}
On the other hand, by definition the restriction $\nabla_{\sph}^2 H_N(\bsig)|_{U^{\perp}_{\iota}(\bsig;H_N)\times U^{\perp}_{\iota}(\bsig;H_N)}$ has operator norm at most $C$ and all eigenvalues outside $[-3\iota,3\iota]$.
It follows from the event $\Sbdd^{1.6}(H_N,\wt H_N)$ that $\|\nabla_{\sph}^2 \wt H_N(\bsig) -\nabla_{\sph}^2 H_N(\bsig)\|_{\op}\leq O(\sqrt{\eps})\ll \iota$.
Hence $\nabla^2 F(0)|_{U^{\perp}_{\iota}(\bsig;H_N)\times U^{\perp}_{\iota}(\bsig;H_N)}$ has operator norm at most $2C$ and all eigenvalues outside $[-2\iota,2\iota]$.

In light of Proposition~\ref{prop:gradients-bounded} and \eqref{eq:u-approx-stationary}, a quantitative inverse function theorem around $u$ (ensuring locally quadratic convergence of Newton's method as in e.g.\ \cite[Theorem 5.3.2]{stoer2013introduction}) now implies existence of a suitable $y\in U_{\iota}^{\perp}(\bsig;H_N)$ with $\|y-u\|\leq O_{d,\iota}(\delta\sqrt{N})$, such that $P^{\perp}\nabla F(y)=0$. 
Setting $\wt \bsig(u)=T_{\bsig}(y)$ yields \eqref{eq:restricted-stationary-point}, completing the proof of existence.

Next we show $\wt\bsig(u)$ is unique.
As above, all eigenvalues of $\nabla^2 F(y)|_{U^{\perp}_{\iota}(\bsig;H_N)\times U^{\perp}_{\iota}(\bsig;H_N)}$ are outside $[-2\iota,2\iota]$.
Supposing for sake of contradiction that $y$ is not unique, let $y'\neq y$ satisfy $y'-y\in U_{\iota}^{\perp}(\bsig;H_N)$ and $\|y'-y\|\leq O_{d,\iota}(\delta\sqrt{N})$.
Then
\begin{align}
\notag
\|P^{\perp}\nabla F(y')\|
&\geq 
\|\nabla^2 F(y)[y'-y]\|-O(\iota\|y'-y\|^2)
\\
\notag
&\geq 
\iota \|y'-y\|-O(\iota\|y'-y\|^2)
\\
\label{eq:explicit-LB}
&\geq
\iota \|y'-y\|/2
\\
\notag
&>0.
\end{align}
This is a contradiction, concluding the proof of uniqueness.
\end{proof}

\begin{proof}[Proof of Lemma~\ref{lem:state-following-local-lipschitz-1-step}]
Suppose $H_N^*,\wt H_N^*,\bsig^*,u^*$ are respectively within some small $\beta_N\sqrt{N}$ of $H_N,\wt H_N,\bsig,u$.
(Here $\beta_N$ may be arbitrarily small, even depending on $N$.)
Write $\wt\bsig=\wt\bsig(u)$ and let $\wt\bsig^*$ be the corresponding point for $(H_N^*,\wt H_N^*,\bsig^*,u^*)$.
By Proposition~\ref{prop:davis-kahan}\ref{it:davis-kahan}, it follows that there are orthonormal bases $v_1,\dots,v_d$ and $w_1,\dots,w_d$ for respectively $U_{\iota}(\bsig;H_N)$ and $U_{\iota}(\bsig^*;H_N^*)$ such that $\|v_i-w_i\|\leq O_{d,\iota}(\beta_N)$ for each $1\leq i\leq d$.
Consider the functions $F_3,F_3^*:\bbR^d\cap\cB_N\to\bbR$ defined by:
\[
F_3(z)
=
\wt H_N\Big(T_{\bsig}\Big(u+\sum_{i=1}^d z_i v_i\Big)\Big),
\quad\quad
F_3^*(z)
=
\wt H_N^*\Big(T_{\bsig^*}\Big(u^*+\sum_{i=1}^d z_i w_i\Big)\Big).
\]
We claim that $F_3$ and $F_3^*$ are close in $C^2$ in the sense that for each $\|z\|<\sqrt{N}$ and $k\in \{0,1,2\}$ we have 
\begin{equation}
\label{eq:C3-approx}
\|\nabla^k F_3(z)-\nabla^k F_3^*(z)\|_{\op}
\leq 
O(\beta_N N^{1-\frac{k}{2}}).
\end{equation}
This follows from Proposition~\ref{prop:geodesics-smooth}: smoothness of $\nabla^{k+1}$ implies Lipschitzness of $\nabla^k$.
For instance in the $k=0$ case,
\begin{align*}
    \|F_3(z)-F_3^*(z)\|
    &\leq 
    \Big(\sup_{x\in\cS_N}
    \|\wt H_N(x)-\wt H_N^*(x)\|_{\op}
    \Big)
    \\
    &+
    \Big\|\wt H_N\Big(T_{\bsig}\Big(u+\sum_{i=1}^d z_i v_i\Big)\Big)
    -
    \wt H_N\Big(T_{\bsig^*}\Big(u+\sum_{i=1}^d z_i v_i\Big)\Big)\Big\|
    \\
    &+
    \Big\|\wt H_N\Big(T_{\bsig^*}\Big(u+\sum_{i=1}^d z_i v_i\Big)\Big)
    -
    \wt H_N\Big(T_{\bsig^*}\Big(u+\sum_{i=1}^d z_i w_i\Big)\Big)\Big\|
    .
\end{align*}
The first term is controlled by Proposition~\ref{prop:gradstable} with $k=1$, the second by Proposition~\ref{prop:geodesics-smooth}, and the third by combining Proposition~\ref{prop:geodesics-smooth} and Lemma~\ref{lem:state-following-adjust-basis}.
The cases $k=1,2$ are similar, using the chain rule to differentiate through the compositions.

Next we recall that $\wt\bsig(u)$ was constructed as $T_{\bsig}(y)$ where $y=u+z$ for $z\in U_{\iota}^{\perp}$ the locally unique and $\Omega_{d,\iota}(1)$-well conditioned stationary point of $z\mapsto \wt H_N(T_{\bsig}(u+z))$.
If \eqref{eq:C3-approx} holds for small $\beta_N$ depending on $(d,\iota)$, then by e.g.\ \cite[Theorem 5.3.2]{stoer2013introduction} the function $z^*\mapsto \wt H_N^*(T_{\bsig^*}(u^*+z^*))$ will have a locally unique well-conditioned stationary point $z^*$ with $\|z-z^*\|\leq O_{d,\iota}(\beta_N\sqrt{N})$.
Since Proposition~\ref{prop:geodesics-smooth} implies that $T$ is locally $O(1)$ Lipschitz in both arguments, it follows that 
\[
\|\wt\bsig-\wt\bsig^*\|
\leq 
\|T_{\bsig}(u+z)-T_{\bsig^*}(u^*+z^*)\|
\leq 
O(\|\bsig-\bsig^*\|+\|u-u^*\|+\|z-z^*\|)
\leq 
O_{d,\iota}(\beta_N\sqrt{N}).\qedhere
\]
\end{proof}

\begin{proof}[Proof of Lemma~\ref{lem:state-following-adjust-basis}]
Let $P=P_{U_{\iota}(\bsig;H_N)}$ and $\wt P=P_{U_{\iota}(\wt\bsig;\wt H_N)}$
By the Davis-Kahan theorem from Proposition~\ref{prop:davis-kahan}, we have:
\[
\|P-\wt P\|
\leq 
O\lt(\frac{\|\nabla_{\sph}^2 H_N(\bsig)-\nabla_{\sph}^2 \wt H_N(\wt\bsig)\|}{\iota\sqrt{N}}\rt)
\leq 
O\lt(\frac{\|H_N-\wt H_N\|}{\iota\sqrt{N}}\rt).
\]
(Here in the latter step we used Proposition~\ref{prop:gradients-bounded}, and that $u\mapsto\wt\bsig$ is locally Lipschitz by Lemma~\ref{lem:state-following-local-lipschitz-1-step}.)
We apply the Gram-Schmidt process to $(\wt P v_1,\dots,\wt P v_d)$ to arrive at an orthonormal basis $(w_1,\dots,w_d)$ for $U_{2\iota}(\wt\bsig(u);H_N^{(1)})$, with $\spann(\wt P v_1,\dots,\wt P v_j)=\spann(w_1,\dots,w_j)$ for all $j$.
It is well-known that the joint dependence of $(w_i)$ on $(\wt P v_i)$ is $O(1)$-Lipschitz in a neighborhood of any orthonormal basis, see e.g. \cite[Section 19.9]{higham2002accuracy}.
Since the Gram-Schmidt process is $d$-dimensional, this Lipschitz constant is independent of $N$.
Since compositions of locally Lipschitz functions are locally Lipschitz, the proof is complete.
\end{proof}

\begin{proof}[Proof of Lemma~\ref{lem:state-following-works-1-step}]
Let $y'=T_{\bsig}^{-1}(\cA(\wt H_N))$ and $y=T_{\bsig}^{-1}(\wt\bsig(u))$.
Since $T_{\bsig}$ is a diffeomorphism with uniformly bounded norm and inverse norm on $\|u\|\leq \sqrt{N}$ (independently of $N$), the condition $\Sall(H_N,\wt H_N;\cA_N)$ implies that $\|\nabla F(y')\|\leq O(\|\nabla\wt H_N(\wt\bsig(u))\|)\leq O(\delta)$.
(Here $F$ is as in \eqref{eq:F-def}.)
On the other hand, the lower bound \eqref{eq:explicit-LB} yields
\[
\|\nabla F(y')\|\geq \iota \|y'-y\|/2
\quad\implies\quad
\|y'-y\|\leq O(\delta).
\]
Applying $T_{\bsig}$ (which is $O(1)$-Lipschitz by Proposition~\ref{prop:geodesics-smooth}) to $y,y'$ yields the first line of \eqref{eq:state-following-works}.
The second line also follows (after adjusting $C'=\delta'/\delta$) since $\wt H_N\in K_N$ and $\|\nabla_{\sph} \wt H_N(\cA_N(\wt H_N))\|\leq\delta\sqrt{N}$.
\end{proof}

\subsection*{Acknowledgement}

Thanks to Brice Huang for helpful comments.

\footnotesize
\bibliographystyle{alphaabbr}
\bibliography{bib}

\end{document}

%% file: commands.tex
\newcommand{\bea}{\begin{eqnarray}}
\newcommand{\eea}{\end{eqnarray}}
\newcommand{\<}{\langle}
\renewcommand{\>}{\rangle}

\newcommand{\wt}{\widetilde}
\newcommand{\op}{\text{op}}
\newcommand{\wh}{\widehat}

\def\bdd{{\rm bounded}}
\def\solve{{\rm solve}}
\def\all{{\rm all}}

\newcommand\eg{{\text{\eg~}}}

\def\eps{{\varepsilon}}

\def\bg{{\boldsymbol{g}}}

\def\bx{{\boldsymbol{x}}}

\def\cT{{\mathcal T}}

\def\op{{\rm op}}

\def\bsig{{\boldsymbol {\sigma}}}

\def\brho{{\boldsymbol \rho}}

\def\by{{\boldsymbol y}}

\def\bx{{\boldsymbol{x}}}

\def\bA{\boldsymbol{A}}

\def\de{{\rm d}}

\def\<{\langle}
\def\>{\rangle}

\def\cN{{\mathcal N}}

\def\cL{{\mathcal L}}

\def\by{{\boldsymbol{y}}}

\def\b0{{\boldsymbol{0}}}

\def\spec{{\rm spec}}

\def\rd{{\mathrm {rad}}}
\def\spann{{\mathrm {span}}}

\def\bG{{\boldsymbol G}}

\DeclareMathOperator*{\plim}{p-lim}

\def\OPT{{\sf OPT}}

\def\cA{{\mathcal A}}

\def\cS{{\mathcal S}}

\def\cB{{\mathcal B}}

\renewcommand{\b}{\mathbf{b}}

\def\lt{\left}
\def\rt{\right}

\def\la{\langle}
\def\ra{\rangle}

\def\eps{\varepsilon}

\def\bbE{{\mathbb{E}}}

\def\bbN{{\mathbb{N}}}
\def\bbP{{\mathbb{P}}}
\def\bbR{{\mathbb{R}}}

\def\cA{{\mathcal{A}}}
\def\cB{{\mathcal{B}}}

\def\cN{{\mathcal{N}}}

\def\sH{{\mathscr{H}}}

\def\bg{{\mathbf{g}}}

\def\ALG{{\mathsf{ALG}}}

\def\OPT{{\mathsf{OPT}}}

\def\Ssolve{S_{\mathrm{solve}}}
\def\hSsolve{\wh S_{\mathrm{solve}}}
\def\psolve{p_{\mathrm{solve}}}
\def\punstable{p_{\mathrm{unstable}}}
\def\Sstab{S_{\mathrm{stab}}}

\def\Sall{S_{\mathrm{all}}}

\def\hSbdd{\wh S_{\mathrm{bdd}}}
\def\hSall{\wh S_{\mathrm{all}}}

\def\Sbdd{S_{\mathrm{bdd}}}

\def\sph{\mathrm{sp}}

\newcommand{\norm}[1]{{\lt\|#1\rt\|}}
\newcommand{\tnorm}[1]{{\|#1\|}}

\def\Lip{\mathrm{Lip}}
\def\LocLip{\mathrm{LocLip}}